\documentclass[11pt]{amsart}
\usepackage{amssymb}
\usepackage{amsthm}
\usepackage{amsmath}
\usepackage{mathtools}
\usepackage{latexsym}
\usepackage{comment}
\usepackage{hyperref}
\usepackage{cleveref}
\usepackage{verbatim}
\usepackage{xypic}
\usepackage{graphicx}
\usepackage{tikz-cd}
\usepackage[utf8]{inputenc}
\usepackage[margin=1in]{geometry}
\usepackage[shortlabels]{enumitem}
\usepackage{xcolor}

\usetikzlibrary{decorations.pathmorphing}

\newtheorem{thm}{Theorem}[section]
\newtheorem{prop}[thm]{Proposition}

\newtheorem{lemma}[thm]{Lemma}
\newtheorem{cor}[thm]{Corollary}

\theoremstyle{definition}

\theoremstyle{definition} \newtheorem{rmk}[thm]{Remark}

\newcommand{\zz}{\mathbb{Z}}

\newcommand{\aff}{\mathbb{A}}
\newcommand{\proj}{\mathbb{P}}

\graphicspath{{pics/}}

\title{Polynomials whose $n$th powers have prescribed multiple-of-$n$th-degree coefficients}
\author{Jeffrey Yelton}

\begin{document}

\maketitle

\begin{abstract}

We show under a mild hypothesis that given field elements $a_0, \dots, a_m \in K$, there always exists a degree-$m$ polynomial whose $n$th power whose degree-$jn$ coefficient is equal to $a_j$ for $0 \leq j \leq m$.  We provide an alternate proof for the $n = 2$ case which is more constructive.

\end{abstract}

The purpose of this note is to prove the theorem below, which has applications to the author's work on constructing semistable models of superelliptic curves defined by an equation of the form $y^p = f(x)$ over mixed characteristic $(0, p)$.  Throughout, we follow the convention that any polynomial $f$ is considered to have degree-$t$ coefficient equal to $0$ for any integer $t > \deg(f)$.

\begin{thm} \label{thm main}

Let $m \geq 0$ and $n \geq 1$ be integers, and let $K$ be a field whose characteristic $p$ is either $0$ or satisfies $n = p^s n'$ with $p > n' \in \zz$.

Given any elements $a_0, \dots, a_m \in K$, there exists a finite algebraic extension $L / K$ and a polynomial $f(x) \in L[x]$ such that the degree-$jn$ coefficient of $f^n$ equals $a_j$ for $0 \leq j \leq m$.

\end{thm}

Let us fix integers $m \geq 0$ and $n \geq 1$.  It clearly suffices to assume that our field $K$ in which the elements $a_j$ live is algebraically closed and set out to prove that there is a polynomial $f(x) \in K[x]$ satisfying the desired property.  We now show that it also suffices to assume that we have $p = 0$ or $n = n' < p$ by the following argument.  Suppose that the field $K$ has characteristic $p > 0$ and that $n = p^s n'$ as in the statement of the theorem.  For $0 \leq j \leq m$, let $b_j \in K$ be an element satisfying $b_j^{p^s} = a_j$.  Because $K$ has characteristic $p$, we have 
\begin{equation}
(b_0 + b_1 x^{n'} + \dots + b_m x^{mn'})^{p^s} = (b_0)^{p^s} + (b_1 x^{n'})^{p^s} + \dots + (b_m x^{mn'})^{p^s} = a_0 + a_1 x^n + \dots + a_m x^{mn}.
\end{equation}
If we assume for the moment \Cref{thm main} holds over $K$ for any exponent $n < p$, replacing our fixed $n = p^s n'$ by $n'$ and the elements $a_0, \dots, a_m \in K$ by $b_0, \dots, b_m \in K$ in the statement of the theorem, we get a polynomial $f(x) \in K[x]$ such that the degree-$jn'$ coefficient of $f^{n'}$ equals $b_j$ for $0 \leq j \leq m$.  Since taking the $p^s$th power of a polynomial over $K$ amounts to adding $p^s$th powers of each of the terms of the polynomial, for each $j$, the degree-$(jn'p^s = jn)$ coefficient of $(f^{n'})^{p^m} = f^n$ equals $b_j^{p^s} = a_j$, so $f$ is the polynomial satisfying the desired conclusion for $n$ and $a_0, \dots, a_m \in K$.

In light of the above argument, we assume from now on that we have $p = 0$ or $n < p$.

The author is grateful to Robert Lemke Oliver for suggesting to him the following method of proof.  For any integers $m \geq 0$ and $n \geq 1$, we define the map $\varphi_{m, n} : \aff_K^{m+1} \to \aff_K^{m+1}$ as the one taking an $(m+1)$-tuple $(\alpha_0, \dots, \alpha_m) \in K^{m+1}$ to the $(m+1)$-tuple $(a_0, \dots, a_m) \in K^m$, where each $a_j$ is the degree-$jn$ coefficient of the polynomial $(\alpha_0 + \alpha_1 x + \dots + \alpha_m x^m)^n \in K[x]$.  Proving \Cref{thm main} clearly amounts to showing that $\varphi_{m, n}$ is surjective.

We observe that the map $\varphi_{m, n}$ is given by an $(m+1)$-tuple of polynomials in the variables $\alpha_0, \dots, \alpha_m$ and that these polynomials are each homogeneous of degree $n$; equivalently, it is immediate from the definition of $\varphi_{m, n}$ that given any scalar $\gamma \in K$, we have 
\begin{equation} \label{eq homogeneity}
\varphi_{m, n}(\gamma (\alpha_0, \dots, \alpha_m)) = \gamma^n \varphi_{m, n}(\alpha_0, \dots, \alpha_m),
\end{equation}
 where the notation on each side indicates scaling of an $(m+1)$-tuple in $K^{m+1}$ treated as a vector.  It follows that $\varphi_{m, n} : \aff_K^{m+1} \to \aff_K^{m+1}$ is a morphism and that it induces a morphism $\proj_K^m \to \proj_K^m$, which by slight abuse of notation we also denote by $\varphi_{m, n}$.

\begin{lemma} \label{lemma morphism conditions}

With the above set-up, the following are equivalent.

\begin{enumerate}[(i)]

\item The morphism $\varphi_{m, n} : \aff_K^{m+1} \to \aff_K^{m+1}$ is surjective.

\item The morphism $\varphi_{m, n} : \proj_K^{m+1} \to \proj_K^{m+1}$ is surjective.

\item The morphism $\varphi_{m, n} : \proj_K^{m+1} \to \proj_K^{m+1}$ is dominant.

\end{enumerate}

\end{lemma}

\begin{proof}
Let $\boldsymbol{0}$ denote the origin in $\aff_K^{m+1}$, and note that it is fixed by $\varphi_{m, n}$.  Since the obvious map $\aff_K^{m+1} \smallsetminus \{\boldsymbol{0}\} \to \proj_K^m$ is surjective, it is immediate that (i) implies (ii).  Now assume that (ii) holds and choose a $K$-point of $\aff_K^{m+1}$ viewed as an $(m+1)$-tuple $\boldsymbol{\beta} \in K^{m+1}$.  The $K$-point $P \in \proj_K^m$ that it reduces to is the image under $\varphi_{m, n}$ of some $K$-point of $\proj_K^m$, which means that it is represented by an $(m+1)$-tuple $\boldsymbol{\beta}' \in K^{m+1} \smallsetminus \{\mathbf{0}\}$ which is the image under $\varphi_{m, n}$ of an $(m+1)$-tuple $\boldsymbol{\alpha} \in K^{m+1} \smallsetminus \{\mathbf{0}\}$.  Now since $\boldsymbol{\beta}, \boldsymbol{\beta}' \in K^{m+1}$ represent the same point of $\proj_K^m$, we must have $\boldsymbol{\beta}' = \gamma \boldsymbol{\beta}$ for some $\gamma \in K^\times$.  Letting $\delta \in K^\times$ be an $n$th root of $\gamma$, using (\ref{eq homogeneity}), we then have 
\begin{equation}
\varphi_{m, n}(\delta^{-1}\boldsymbol{\alpha}) = \gamma^{-1}\boldsymbol{\beta}' = \boldsymbol{\beta}.
\end{equation}
This shows that $\boldsymbol{\beta}$, viewed as a $K$-point of $\aff_K^{m+1}$, lies in the image of $\varphi_{m, n}$, and so $\varphi_{m, n} : \aff_K^{m+1} \to \aff_K^{m+1}$ is surjective.  Thus, properties (i) and (ii) are equivalent.

The fact that (ii) implies (iii) is immediate.  The converse is a direct corollary of the well-known fact that morphisms from projective spaces are closed and thus their images are closed: see for instance \cite[\S I.5, Theorem 2]{shafarevich1994basic}.
\end{proof}

Now the key to proving that the equivalent conditions treated by \Cref{lemma morphism conditions} all hold is to show that the \emph{Jacobian matrix} of the morphism $\varphi_{m, n} : \aff_K^{m+1} \to \aff_K^{m+1}$ is invertible at some point in the domain.  This is done through the following proposition.

\begin{prop} \label{prop Jacobian}

In the above situation, assume that we have $\gcd(m, n) = 1$.  Writing $\alpha_0, \dots, \alpha_m$ for the coordinate variables of $\aff_K^{m+1}$ and writing $a_i$ for the $i$th coordinate of $\varphi_{m, n}(\alpha_0, \dots, \alpha_m)$, the Jacobian matrix $\big(\frac{\partial a_{i-1}}{\partial \alpha_{j-1}}\big)_{1 \leq i, j \leq m + 1}$ is invertible at the point $(\alpha_0, \dots, \alpha_m) = (1, 0, \dots, 0, 1)$ in the domain.

\end{prop}

\begin{proof}
The $(i, j)$th entry of the Jacobian is given by the degree-$(i-1)n$ term of the polynomial 
\begin{equation}
\begin{aligned}
\tfrac{\partial}{\partial \alpha_{j-1}}(\alpha_m x^m + \dots + \alpha_1 x + \alpha_0)^n &= n(\alpha_m x^m + \dots + \alpha_1 x + \alpha_0)^{n-1} \tfrac{\partial}{\partial \alpha_{j-1}} (\alpha_{j-1} x^{j-1}) \\
&= n(\alpha_m x^m + \dots + \alpha_1 x + \alpha_0)^{n-1} x^{j-1}.
\end{aligned}
\end{equation}
Evaluated at $(\alpha_0, \dots, \alpha_m) = (1, 0, \dots, 0, 1)$, we get 
\begin{equation}
\begin{aligned}
\tfrac{\partial}{\partial \alpha_{j-1}}(\alpha_m x^m + \dots + \alpha_1 x + \alpha_0)^n \big|_{(\alpha_0, \dots, \alpha_m) = (1, 0, \dots, 0, 1)} &= n(x^m + 1)^{n-1} x^{j-1} \\
&= \sum_{k = 0}^{n-1} n {n - 1 \choose k} x^{mk+j-1}.
\end{aligned}
\end{equation}
Denoting the reduction of any integer $a \in \zz$ modulo $n$ by $\bar{a} \in \zz / n\zz$, we note that the element $\overline{m} \in \zz / n\zz$ has a multiplicative inverse $\overline{m}^{-1} \in \zz / n\zz$ since the integers $m$ and $n$ are relatively prime.  Fixing an index $j$, we thus have a unique element $\bar{k} \in \zz / n\zz$ such that we have $\overline{m}\bar{k} + \bar{j} - \bar{1} = \bar{0} \in \zz / n\zz$, given by $\bar{k} = \overline{m}^{-1}(\bar{1} - \bar{j})$.  There is then a unique $k \in \{0, \dots, n - 1\}$ reducing modulo $n$ to $\bar{k}$, which is the only $k$ in this interval satisfying $mk + j - 1 \in n\zz$.  The inequalities $0 \leq k \leq n - 1$ and $1 \leq j \leq m + 1$ imply that $0 \leq mk + j - 1 \leq mn$, and so we have $0 \leq \frac{1}{n}(mk + j - 1) \leq m$; this gives us $mk + j - 1 = (i - 1)n$ for some unique $i \in \{1, \dots, m + 1\}$.  As we have $p = 0$ or $n < p$, the coefficient $n {n - 1 \choose k}$ is nonzero; it follows that there is a unique nonzero entry in the $j$th column of the Jacobian matrix.

Now suppose that there are two indices $j, j'$ such that the $(i, j)$th and $(i, j')$th entries of the Jacobian matrix are both nonzero.  Then we have integers $k, k' \in \{0, \dots, n - 1\}$ satisfying 
\begin{equation} \label{eq j j'}
mk + j - 1 = (i - 1)n = mk' + j' - 1.
\end{equation}
This implies that $m(k - k') = j' - j$.  Suppose that we have $j' \neq j$.  Then the only way for $m$ to divide the difference $j' - j$ is if $j' = m + 1$ and $j = 1$, which gives us $m(k - k') = (m + 1) - 1 = m$ so that $k - k' = 1$.  Now from the first equation in (\ref{eq j j'}), we get $mk = (i - 1)n$.  Since $m$ and $n$ are relatively prime, this means that $n \mid k$, which forces $k = 0$.  This is a contradiction, since $k' = k - 1$ cannot be negative.  Therefore, we have $j' = j$, which is to say that for each index $i$, the $i$th row of the Jacobian matrix has a unique nonzero entry.

Now since each row and each column of the Jacobian matrix at the point $(\alpha_0, \dots, \alpha_m) = (1, 0, \dots, 0, 1)$ has a unique nonzero entry, that matrix is invertible (being a permutation matrix times a diagonal matrix with nonzero diagonal entries).
\end{proof}

\begin{cor} \label{cor surjectivity if relatively prime}

In the above situation, assume that we have $\gcd(m, n) = 1$.  Then the map $\varphi_{m, n} : \aff_K^{m+1} \to \aff_K^{m+1}$ is surjective.

\end{cor}

\begin{proof}
\Cref{prop Jacobian} implies that the morphism $\varphi_{m, n}$ is \'{e}tale at the point $(\alpha_0, \dots, \alpha_m) = (1, 0, \dots, 0, 1)$ (as in \cite[Example 5.45]{milne2024algebraic}); it is therefore \'{e}tale on a Zariski open subvariety $U \subset \aff_K^{m+1}$ (see for instance \cite[Lemma 5.54]{milne2024algebraic}).  Since $\varphi_{m, n}$ is an \'{e}tale morphism of $U$ onto its image $\varphi_{m, n}(U)$, and since $U$ has dimension $m + 1$, its image $\varphi_{m, n}(U)$ must have dimension $m + 1$ as well.  Since the Zariski closure of any dimension-$(m+1)$ subspace of $\aff_K^{m+1}$ is the whole space $\aff_K^{m+1}$, the image of $\varphi_{m, n}$ is dense in $\aff_K^{m+1}$; in other words, the morphism $\varphi_{m, n} : \aff_K^{m+1} \to \aff_K^{m+1}$ is dominant.

Now the obvious (surjective) map $\aff_K^{m+1} \smallsetminus \{\boldsymbol{0}\} \to \proj_K^m$ (as in the proof of \Cref{lemma morphism conditions}) is easily verified to be continuous under the Zariski topology; it is an elementary exercise in point-set topology to show that the image of a dense subspace under a continuous map is dense in the image of that map, so we have that the image of $\varphi_{m, n}(\aff_K^{m+1}) \subset \aff_K^{m+1}$ in $\proj_K^m$ is a dense subspace of $\proj_K^m$ as well.  Thus, the map $\varphi_{m, n} : \proj_K^m \to \proj_K^m$ is also dominant.  Then by \Cref{lemma morphism conditions}, the map $\varphi_{m, n} : \aff_K^{m+1} \to \aff_K^{m+1}$ is surjective.
\end{proof}

The following corollary finishes the proof of \Cref{thm main}.

\begin{cor}

In the above situation, for any integers $m \geq 0$ and $n \geq 1$ (not necessarily relatively prime), the map $\varphi_{m, n} : \aff_K^{m+1} \to \aff_K^{m+1}$ is surjective.

\end{cor}

\begin{proof}
If we have $\gcd(m, n) = 1$, then we are done by \Cref{cor surjectivity if relatively prime}, so let us assume that $\gcd(m, n) > 1$.  In that case, we certainly have $\gcd(m + 1, n) = 1$, and then \Cref{cor surjectivity if relatively prime} tells us that the morphism $\varphi_{m+1, n} : \aff_K^{m+2} \to \aff_K^{m+2}$ is surjective.  Choose any $(m+1)$-tuple $(a_0, \dots, a_m) \in K^{m+1}$.  There is an $(m+2)$-tuple $(\alpha_0, \dots, \alpha_{m+1}) \in K^{m+2}$ such that we have $\varphi_{m+1, n}(\alpha_0, \dots, \alpha_{m+1}) = (a_0, \dots, a_m, 0)$.  By definition of $\varphi_{m+1, n}$, the polynomial $f(x) := \alpha_0 x + \dots + \alpha_{m+1} x^{m+1}$ satisfies that the polynomial $f^n$ has no degree-$(m+1)n$ term.  Since we have $\deg(f) \leq m + 1$, we also have $\deg(f^n) = n\deg(f) \leq (m + 1)n$, and by what we have just observed, this inequality is strict; therefore we have the strict inequality $\deg(f) < m + 1$ as well.  In other words, we have $\alpha_{m+1} = 0$, so that $f(x) = \alpha_0 + \dots + \alpha_m x^m$.  By construction, the degree-$jn$ coefficient of $f^n$ equals $a_j$ for $0 \leq j \leq m$, meaning that $\varphi_{m, n}(\alpha_0, \dots, \alpha_m) = (a_0, \dots, a_m)$.  This proves the desired surjectivity.
\end{proof}

We finish by including a completely independent proof of \Cref{thm main} in the case that $n = 2$ (noting that the hypothesis on the characteristic of $K$ now always holds), which is of a more constructive nature but which does not generalize in any obvious way to $n \geq 3$.  This is adapted from the author's collaborated preprint \cite{fiore2023clusters}, in which it appears in the form of Proposition 4.20.  The author would like to credit Leonardo Fiore for having originally provided this proof.

\begin{proof}[Proof of \Cref{thm main} in the $n = 2$ case]
The case of $a_0 = \cdots = a_m = 0$ is obvious, so we may assume there is a greatest integer $m' \in \{0, \dots, m\}$ such that $a_{m'} \neq 0$.  In this case, a polynomial satisfying the desired condition with respect to the elements $a_0, \dots, a_{m'} \in K$ satisfies it for $a_0, \dots, a_m \in K$, so, after replacing $m$ with $m'$, we assume from now on that $a_m \neq 0$.  Letting $\sqrt{a_m} \in K$ be a square root of $a_m$, if $f(x) \in K[x]$ is a polynomial satisfying the desired condition with respect to $a_m^{-1} a_0, \dots, a_m^{-1} a_{m-1}, 1 \in K$, then it is clear that $\sqrt{a_m}f$ satisfies the desired condition with respect to $a_0, \dots, a_{m-1}, a_m \in K$, so, after replacing $a_j$ with $a_m^{-1} a_j$ for $0 \leq j \leq m$, we assume from now on that $a_m = 1$.

Let the roots of the polynomial $a_0 + a_1 x + \dots + x^m$ be denoted (with multiplicity) by $\beta_1, \dots, \beta_m \in K$.  For each $\beta_i$, choose a square root $\sqrt{\beta_i} \in K$.  It is clear that the multiset of roots of the polynomial $g(x) := a_0 + a_1 x^2 + \dots + x^{2m}$ (counted with multiplicity) is $\{\pm\sqrt{\beta_i}\}_{1 \leq i \leq m}$, and that we have $g = h_+ h_-$, where 
\begin{equation*}
	\begin{split}
		h_+(z) &:= \prod_i (z+\sqrt{\beta_i}) = c_0 + c_1 x + \dots + x^m,\\
		h_-(z) &:= \prod_i (z-\sqrt{\beta_i}) = (-1)^m c_0 + (-1)^{m-1} c_1 x + \dots + x^m.\\
	\end{split}
\end{equation*}
For $0 \leq k \leq 2m$, the degree-$k$ coefficient of $g$ is given by $\sum_{i+j=k} (-1)^{m-i} c_i c_j$.

Let us now choose a square root $\sqrt{-1} \in K$ of $-1$ and define 
\begin{equation*}
	c'_i := \begin{cases}
		c_i & \text{if $m \equiv i$ (mod $2$)},\\
		\sqrt{-1}\cdot c_i & \text{if $m \not\equiv i$ (mod $2$)}.
	\end{cases}
\end{equation*}
Then for each even value of $k$, we may rewrite our formula for the degree-$k$ coefficient of $g$ in the more symmetric form $\sum_{i+j=k} c'_i c'_j.$  Now letting $f(x) = c'_0 + c'_1 x + \dots + c'_m x^m$, it is clear that the even-degree terms of $[f(x)]^2$ add up to $g(x)$.  We have thus proved \Cref{thm main} for $n = 2$.
\end{proof}

\begin{rmk}

One notes in the case of $n = 2$ that the above proof furnishes $2^{m+1}$ polynomials $f$ satisfying the desired conclusion, coming from the $m + 1$ (independent) choices of sign of the elements $\sqrt{a_m}, \sqrt{\beta_1}, \dots, \sqrt{\beta_m} \in K$, provided $K$ does not have characteristic $2$ and that the elements $\beta_1, \dots, \beta_m$ are all distinct (in other words, provided the polynomial $a_0 + a_1 x + \dots + x^m$ has no multiple root).  For general $n$ and for a ``general" choice of elements $a_0, \dots, a_m \in K$, assuming that $K$ has characteristic either $0$ or greater than $n$, we suspect that there are exactly $n^{m+1}$ polynomials $f$ satisfying the property claimed by \Cref{thm main}, or in other words, that the degree of the morphism $\varphi_{m, n}$ equals $n^{m+1}$.

One can show, using arguments analogous to those in the proof of \Cref{prop Jacobian}, that a partial converse of \Cref{prop Jacobian} also holds: if we have $\gcd(m, n) > 1$, then the Jacobian matrix is \textit{not} invertible at $(\alpha_0, \dots, \alpha_m) = (1, 0, \dots, 0, 1)$.  This implies that in this case, the morphism $\varphi_{m, n}$ does not ramify at $(1, 0, \dots, 0, 1)$.  Note that when $n \mid m$, letting $(a_0, \dots, a_m) = \varphi_{m, n}(1, 0, \dots, 0, 1)$, we have $a_0 + a_1 x + \dots + a_m x^m = (1 + x^{m/n})^n$, a polynomial with a multiple root (although when $\gcd(m, n) > 1$ and $n \nmid m$, this is generally not true).  It seems plausible that for any $m \geq 0$ and $n \geq 1$, given elements $\alpha_0, \dots, \alpha_m$ and letting $(a_0, \dots, a_m) = \varphi_{m, n}(\alpha_0, \dots, \alpha_m)$, the inseparability of the polynomial $a_0 + a_1 x + \dots + a_m x^m$ may be a necessary (but not sufficient) condition for the map $\varphi_{m, n}$ to ramify at $(\alpha_0, \dots, \alpha_m)$.

\end{rmk}

\bibliographystyle{plain}
\bibliography{bibfile}

\begin{thebibliography}{1}

\bibitem{fiore2023clusters}
Leonardo Fiore and Jeffrey Yelton.
\newblock Clusters and semistable models of hyperelliptic curves in the wild case.
\newblock {\em arXiv preprint arXiv:2207.12490v4}, 2023.

\bibitem{milne2024algebraic}
J.~S. Milne.
\newblock Algebraic geometry. {O}nline lecture notes (v6.10), 2024.

\bibitem{shafarevich1994basic}
Igor Shafarevich.
\newblock {\em Basic algebraic geometry: Varieties in projective space}.
\newblock Springer, 1994.

\end{thebibliography}

\end{document}